\documentclass[a4paper]{article}

\usepackage{marginnote}
\makeatletter
\newcommand{\edit}{\@ifstar\edit@star\edit@nostar}
\newcommand{\edit@star}[1]{{\color{blue}#1}}
\newcommand{\edit@nostar}[2][\@]{%
	\marginnote{%
	\if#1\@
		\"Andringar
	\else
		\"Andring: #1
	\fi
	}{\color{blue}#2}}
\newcommand{\comment}[1]{\bgroup\reversemarginpar\marginnote{\null\hfill Kommentar\hfill\null}\emph{#1}\egroup}
\makeatother

\usepackage[T1]{fontenc}
\usepackage[utf8]{inputenc}
\usepackage[pass]{geometry}
\usepackage{tikz}
	\usetikzlibrary{positioning}
\usepackage{amsmath}
\usepackage{amsfonts}
\usepackage{amssymb}
\usepackage{mathtools}
\usepackage{amsthm}
\usepackage{etoolbox}
\usepackage{caption}
\RequirePackage{hyperref}
	\hypersetup{linktoc=all,hidelinks}
\usepackage[capitalize,nameinlink]{cleveref}
\usepackage[noadjust]{cite}

\theoremstyle{definition}
\newtheorem{theorem}{Theorem}[section]

\newtheorem{lemma}[theorem]{Lemma}
\newtheorem{proposition}[theorem]{Proposition}
\newtheorem{corollary}[theorem]{Corollary}
\newtheorem{property}{Property}[section]

\newtheorem{definition}{Definition}

\newtheorem{problem}{Problem}
\theoremstyle{remark}

\crefname{claim}{Claim}{Claims}
\crefname{case}{Case}{Cases}
\crefname{subcase}{Case}{Cases}
\crefname{conjecture}{Conjecture}{Conjectures}
\AfterEndEnvironment{case}{\noindent\ignorespaces}
\AfterEndEnvironment{subcase}{\noindent\ignorespaces}
\newcounter{parenum}[theorem]

\let\eqref\labelcref
\crefname{equation}{}{}

\crefname{enumi}{}{}



\makeatletter
\DeclareRobustCommand{\crefnocompress}[1]{\begingroup\@cref@compressfalse\cref{#1}\endgroup}
\DeclareRobustCommand{\crefnosort}[1]{\begingroup\@cref@sortfalse\cref{#1}\endgroup}
\makeatother

\newlength{\revdirraise}
\newlength{\revdirextraraise}

\DeclarePairedDelimiterX{\defset}[2]{\lbrace}{\rbrace}
	{\,#1:#2\,}

\newenvironment{textequation*}
	{\begin{equation}\begin{minipage}[b]{0.75\linewidth}
		\setlength{\RaggedRightRightskip}{0em plus 4em minus 2pt}\RaggedRight}
	{\end{minipage}\end{equation}\ignorespacesafterend}

\AtBeginDocument{\frenchspacing}

\newcommand{\internalvertex}[2]{node [circle, inner sep=0pt, outer sep=0pt, minimum size=4pt, fill=#1, draw=#2] {}}
\newcommand{\vertex}{\internalvertex{black}{black}}
\newcommand{\wvertex}{\internalvertex{white}{black}}

\def\upstrut(#1,#2){\node at (0,#2) [anchor=south] {\vphantom{#1}}}
\def\downstrut(#1,#2){\node at (0,#2) [anchor=north] {\vphantom{#1}}}

\newcommand{\figsmallsquare}{%
\begin{tikzpicture}
\foreach \x in {-1,0}{
	\foreach \y in {-1,0}{
		\draw (\x,\y) -- (\x+1,\y)
              (\x,\y) -- (\x,\y+1);}}
\draw (-1,1) -- (1,1) -- (1,-1); 
\draw (-1,0) -- (0,1) -- (1,0) -- (0,-1) -- cycle;
\foreach \x in {-1,0,1}{
	\foreach \y in {-1,0,1}{
		\tikzmath{\z=\x+\y;}
		\ifodd \z
			\draw (\x,\y) \wvertex;
		\else
			\draw (\x,\y) \vertex;
		\fi}}
\end{tikzpicture}}
\newcommand{\figbigsquare}{%
\begin{tikzpicture}
\foreach \x in {-2,...,1}{
	\foreach \y in {-2,...,1}{
		\draw (\x,\y) -- (\x+1,\y)
              (\x,\y) -- (\x,\y+1);}}
\draw (-2,2) -- (2,2) -- (2,-2); 
\draw (-2,-1) -- (1,2) -- (2,1) -- (-1,-2) -- cycle;
\draw (2,-1) -- (-1,2) -- (-2,1) -- (1,-2) -- cycle;
\foreach \x in {-2,...,2}{
	\foreach \y in {-2,...,2}{
		\tikzmath{\z=\x+\y;}
		\ifodd \z
			\draw (\x,\y) \wvertex;
		\else
			\draw (\x,\y) \vertex;
		\fi}}
\end{tikzpicture}}

\begin{document}

\title{A note  on  transformations of edge colorings of  chordless graphs and triangle-free  graphs}
\author{Armen  Asratian%
	\footnote{Department of Mathematics, Linköping University, email: armen.asratian@liu.se}
	}
\date{}
\maketitle

\begin{abstract}
\noindent 
Bonamy et al. (2023)
proved that an optimal edge coloring of a simple  triangle--free  graph $G$ can be
 reached from any given proper edge coloring of $G$ through a series of Kempe changes.
 We show that a small modification of their proof  gives a possibility to obtain a similar result 
 for   a larger class  of simple graphs consisting of all  triangle-free and all chordless graphs
 (a graph $G$ is chordless if in every cycle $C$ of $G$  any two nonconsecutive vertices of $C$ are not adjacent).  
\end{abstract}

\bgroup
\noindent
\setlength{\parfillskip}{0pt}%
\par
\egroup

\section{Introduction}

This note is inspired by the paper of Bonamy et al. \cite{bonamy23} which is devoted to transformations of edge colorings of triangle-free graphs.

A $t$-edge coloring or simply $t$-coloring of a graph $G=(V(G),E(G))$ is a
mapping
$\alpha: E(G)\longrightarrow  \{1,...,t\}$.
The set of edges of color $j$, denoted  by $M(\alpha,j)$, is called a {\it color class}, $j=1,..,t$.  
A $t$-coloring of $G$ is
called
proper  if no adjacent edges receive the same color. 
The minimum number $t$ for which there exists a  proper $t$-coloring of $G$ is called 
the {\em chromatic index} of  $G$ and is denoted by $\chi'(G)$. 
Throughout this paper, we only consider proper edge colorings, and so we will only write $t$-colorings to denote proper $t$-edge colorings.

The following result was obtained by Vizing \cite{vizing2}:

\begin{theorem}
  For every   simple graph $G$,
$\chi'(G)\le \Delta(G)+1$, where $\Delta(G)$ denotes the maximum   degree of the vertices of $G$.
\end{theorem}

This result implies that for any  simple graph $G$, either $\chi'(G)=\Delta(G)$ or $\chi'(G)=\Delta(G)+1$. In the former case
	$G$ is said to be \emph{Class $1$}, and in the latter $G$ is {\em Class $2$}. Kempe changes play key role in the proof of Theorem 1.1.

Given a properly edge-colored graph, a {\it Kempe chain} is a maximal connected bicolored subgraph. A $K$-change (short for
{\it Kempe change}) 
corresponds
 to selecting a Kempe chain and swapping the two colors in it. 
  Two  $t$-colorings of a graph $G$ are $K$-equivalent (short for {\it Kempe-equivalent})  if one can be reached 
 from the other through a series of $K$-changes using colors from the set $\{1,...,t\}$.

Vizing \cite{vizing2} showed that any    
$t$-coloring of a simple graph $G$, $t>\Delta(G)+1$, can be transformed to a   
$(\Delta(G)+1)$-coloring of $G$
by using $K$-changes only. He also posed the following problem:

	\begin{problem}
\label{prob:vizing}
Is it true that any Class 1 graph $G$ satisfies the following property:
	every  $t$-coloring of $G$,  $t>\Delta(G)$, can be transformed to a  $\Delta(G)$-coloring  
of $G$ by
a sequence of $K$-changes? (This property we shall call the Vizing property.)
\end{problem}	

Mohar \cite{mohar07}  solved Problem 1 for any Class 1 graph $G$ in the case  $t\geq \Delta(G)+2$.
	In fact, he proved     that  all $t$-colorings of $G$ are $K$-equivalent if $t\geq \chi'(G)+2$. 
	Mohar also posed the following problem:

\begin{problem}
\label{prob:mohar}
 Is it true that all  $(\chi'(G)+1)$-colorings 
 of a  simple graph $G$ are $K$-equivalent?
\end{problem}

The author and Casselgren \cite{asratian16} showed that in  fact Problems 1 and 2   have the same answer
(see \cite{asratian16}, Theorem 1.1).

It was proved that the answers to Problems 1 and 2 is yes if either  
$\Delta(G)=3$ (McDonald,  Mohar and  Scheide \cite{mohar12}), $\Delta(G)=4$ (Asratian and Casselgren \cite{asratian16}),
 $\Delta(G)\geq 9$ and $G$ is planar (Cranston \cite{cranston}),
or $\Delta(G)\geq 5$ and the subgraph induced by the vertices of degree at least 5 in $G$, is acyclic \cite{asratian16}.
For triangle-free graphs the problems  were resolved to the positive by Bonamy et al \cite{bonamy23}.

In this paper we show that a small modification of the proof in \cite{bonamy23}  gives a possibility to solve Problems 1 and  2 
 for  a larger class of simple graphs consisting of all triangle-free and all chordless graphs
 (a graph is chordless if it does not contain chords for any cycle $C$, where a chord for $C$ is an edge that joins two nonconsecutive
vertices of $C$).
Chordless graphs were first studied independently by Dirac \cite{dirac} and Plummer \cite{plummer}
in connection with  minimally 2-connected graphs. 
(a  2-connected graph $G$ is minimally 2-connected, if for any $e\in E(G)$, the graph 
$G-e$ is not 2-connected). 
It can be easily verified that a graph is minimally 2-connected if and only if it is 2-connected and chordless.

It is known  \cite{machado} that  any simple chordless graph $G$ with $\Delta(G)\geq 3$ is Class 1. 
  Note also that there is an infinite set of chordless graphs that are not triangle-free (see, for example, Proposition 3.1 in Section 3).
 
  We prove here the following theorem:

\begin{theorem} 
Let $G$ be a   Class 1 graph that is triangle-free or  chordless. 
Then all $(\Delta(G)+1)$-colorings of $G$ are $K$-equivalent.
\end{theorem}
 
We essentially follow the proof of \cite{bonamy23} (given for triangle-free graphs) but we apply  their arguments 
directly to the  original graph $G$
while  all proof arguments  in \cite{bonamy23} 
were applied to  a special $\chi'(G)$-regular Class 1  graph 
 that  contains $G$ as an induced subgraph and is triangle-free as $G$ itself. 
 For chordless graphs  such a "method of regularization" may not work because 
 for some chordless graphs $G$
 every 
regular graph   containing $G$ as an induced subgraph, may contain a chord and therefore 
we cannot use  the structure of chordless graphs in our arguments.

Theorem 1.2 implies the following two corollaries.

\begin{corollary}
Let $G$ be a  simple  graph that is triangle-free or  chordless. 
Then all  $(\chi'(G)+1)$-colorings of $G$ are $K$-equivalent.
\end{corollary}

\begin{corollary}
If $G$ is a simple chordless graph  with $\Delta(G)\geq 3$, then all $(\Delta(G)+1)$-colorings of $G$ are $K$-equivalent. 
\end{corollary}

\section{Definitions and preliminary results}

Let  $G$ be a  simple  graph with maximum degree $\Delta$.  Consider 
a  $(\Delta + 1)$-coloring $\alpha$ of $G$. We  use here the notation and terminology  from \cite{bonamy23}.

	For  any two (distinct) colors $c$ and $d$, we denote by $K_\alpha(c, d)$ the subgraph of G induced by the edges colored $c$ or $d$ under coloring $\alpha$.

For a vertex $v \in V(G)$, we say that  a color $c\in \{1,2,...,\Delta+1\}$, {\em is missing} in $\alpha$ at $v$ if 
there is no 
 edge $e$ incident to $v$ with $\alpha(e) = c$. Furthermore, with each vertex $x\in V(G)$ we associate a  color, denoted
$m_\alpha(x)$, that is missing  in $\alpha$ at $x$.

For each  $u\in V(G)$ we define a directed graph $D_u(\alpha)$ on vertex set $\{uw :  w\in N(u)\}$, where a 
vertex $uw$ has a directed edge to $ux$ if $m_\alpha(w) = \alpha(ux)$.
	
\begin{definition}	
An $\alpha$--fan	at a vertex $u$ is a sequence of edges $(ux_0,ux_1,...,ux_{p})$ incident with $u$ 
such that 
$m_\alpha(x_i)=\alpha(ux_{i+1})$ for $i=0,1,...,p-1$, and either the color $m_\alpha(x_{p})$ is missing at $u$,  or 
 $m_\alpha(x_p)\in\{\alpha(ux_0),...,\alpha(ux_{p-1})\}$.
 
 For each edge $uv$ there is the unique $\alpha$--fan $(ux_0,ux_1,...,ux_{p})$ where $ux_0=uv$. This
 $\alpha$-fan we denote  by $X_u(\alpha,v)$. 
 \end{definition}

 Consider an  $\alpha$-fan  $X_u(\alpha,v)=(ux_0,ux_1,...,ux_{p})$.
 The sequence of edges in $X_u(\alpha,v)$ corresponds to a sequence of vertices in $D_u(\alpha)$.
 So we can also refer to $X_u(\alpha,v)$ as a sequence of vertices in $D_u(\alpha)$.
  The following  cases are possible: \medskip

  1) The color $m_\alpha(x_{p})$ is missing  at $u$. Then $X_u(\alpha,v)$ is a path in $D_u(\alpha)$.
  
  2) $p\geq 1$ and $m_\alpha(ux_{p})=\alpha(ux_0)$. Then $X_u(\alpha,v)$ is a cycle in $D_u(\alpha)$. 
  
  3) $p\geq 2$ and 
  $m_\alpha(x_p)\in\{\alpha(ux_1),...,\alpha(ux_{p-1})\}$.
  Then we say that $X_u(\alpha,v)$ is a  {\it comet} in
   $D_u(\alpha)$.

 \begin{definition}
  If  $X_u(\alpha,v)$ is a path or cycle in $D_u(\alpha)$,
  we denote by $X_u^{-1}(\alpha,v)$
the coloring $\alpha'$ obtained from $\alpha$ as follows: $\alpha'(ux_i)=m_\alpha(x_i)$, for  $i=0,1,...,p$, and $\alpha'(e)=\alpha(e)$ for any other edge  $e\in E(G)$.  We put $m_{\alpha'}(x_i)=\alpha(ux_i)$ for  $i=0,1,...,p,$ and  $m_{\alpha'}(z)=m_\alpha(z)$ for every  
$z\notin\{u,x_0,...,x_p\}$. Furthermore, we put $m_{\alpha'}(u)=\alpha(ux_0)$
if $X_u(\alpha,v)$ is a path, and   $m_{\alpha'}(u)=m_\alpha(u)$  if $X_u(\alpha,v)$ is a cycle in $D_u(\alpha)$. 
\end{definition}

\begin{lemma}
Let $X_u(\alpha,v)=(ux_0,ux_1,...,ux_{p})$ be a path in $D_u(\alpha)$.
Then the colorings $\alpha$ and $X_u^{-1}(\alpha,v)$ of $G$ are $K$-equivalent.
\end{lemma}

\begin{proof}
This is evident if $p=0$.
If $p\geq 1$, we can transform the coloring $\alpha$ to the coloring $X_u^{-1}(\alpha,v)$ 
  by a sequence of $K$-changes
  as follows: recolor $ux_p$ with $m_\alpha(x_p)$, then  recolor 
  $ux_{p-1}$ with $m_\alpha(x_{p-1})$, . . . , recolor $ux_1$ with $m_\alpha(x_1)$ and, finally, recolor $ux_0$ with 
  $m_\alpha(x_0)$.   
\end{proof}

\begin{definition}
  For any vertex $u$ and a cycle   $X_u(\alpha,v)=(ux_0,ux_1,...,ux_{p})$  in $D_u(\alpha)$, we say that $X_u(\alpha,v)$
 is {\it saturated} if for every $0\leq i\leq p$, the component of $K_\alpha(\alpha(ux_i), m_\alpha(u))$ containing 
 $u$ also contains 
$x_{i-1}$ (resp. $x_p$ if $i = 0$).
\end{definition} 
 
The next result is obtained by  Bonamy et al. \cite{bonamy23}.  
For the convenience of the reader we give their proof here.

\begin{lemma} [\cite{bonamy23}] 
For any vertex $u$ and non-saturated cycle $X_u(\alpha,v)$  in $D_u(\alpha)$,
the colorings  $\alpha$ and $X_u^{-1}(\alpha,v)$ are $K$-equivalent.
\end{lemma}

\begin{proof}
Let $X_u(\alpha,v) = (ux_0,...,ux_p)$ be a non-saturated cycle in $D_u(\alpha)$. Without loss of generality, we can assume that the component of $K_\alpha(\alpha(ux_0),m_\alpha(u))$ containing $u$ does not contain an edge incident with $x_p$. Furthermore, we have $m_\alpha(x_i) = \alpha(ux_{i+1})$ for every $0\leq i<p$, and 
 $m_\alpha(x_p) = \alpha(ux_0)$. We consider the coloring $\alpha'$ obtained from $\alpha$ by swapping the component $C$ of 
 $K_\alpha(\alpha(ux_0),m_\alpha(u))$ containing $x_p$. By assumption, this has no impact on the colors of the edges incident with $u$.
Therefore we have $m_{\alpha'}(x_i) = m_{\alpha} (x_i)$ for every $0\leq i<p$, as well as $m_\alpha'(u) = m_{\alpha}(u)$.
Since the color $m_\alpha(u)$ is missing now at $x_p$, we can put  
 $m_{\alpha'}(x_p) = m_\alpha(u)$.
Now the sequence $X_u(\alpha', x_0) = (ux_0, ..., ux_p)$ is a path in $D_u(\alpha')$. Hence, by Lemma 2.1,
 the colorings $\alpha'$ and $\alpha''=X_u^{-1}(\alpha',x_0)$
 are $K$-equivalent.
In the coloring $\alpha''$, let $C'$ be the component of $K_{\alpha''}(\alpha(ux_0), m_\alpha(u))$ containing $x_p$.
 We have that  $C'= C\cup \{ux_p\}$, and that $m_{\alpha''}(u) = \alpha(ux_0)$, so it suffices to swap $C'$ to obtain 
 $X_u^{-1}(\alpha,x_0)$. Therefore  $X_u^{-1}(\alpha,x_0)$ is $K$-equivalent to $\alpha''$.
Since $\alpha''$ is $K$-equivalent to $\alpha'$ and  $\alpha'$ is $K$-equivalent to $\alpha$, 
we obtain that $\alpha$ and $X_u^{-1}(\alpha,x_0)$
are $K$-equivalent, as desired.
\end{proof}

Bonamy et al.  obtained  in \cite{bonamy23}  the following result:\medskip

Let $G$ be a  $\Delta$-regular Class 1 graph, and $\alpha$ be a $(\Delta+1)$-coloring of $G$.
For any vertex $u$ and saturated cycle $X_u(\alpha, v)$ in $D_u(\alpha)$, 
if $G$ is  triangle-free,  and if  $X_v (\alpha, u)$  is a cycle in  $D_v(\alpha)$,
  then  $\alpha$ and  
$X^{-1}_u(\alpha,v)$  
are $K$-equivalent.\medskip

Using their  proof  in \cite{bonamy23}, we obtain here the following result:

\begin{lemma} 
Let $G$ be a  chordless or triangle-free simple graph with maximum degree $\Delta$, and $\alpha$ be a $(\Delta+1)$-coloring of $G$.
For any vertex $u$ and saturated cycle $X_u(\alpha, v)$ in $D_u(\alpha)$, 
 if the sequence $X_v (\alpha, u)=(vu,vy_1,...,vy_q)$  is a cycle in  $D_v(\alpha)$ and $q\geq 2$,
  then the colorings $\alpha$ and  
$X^{-1}_u(\alpha,v)$  
are $K$-equivalent.
\end{lemma}

\begin{proof}
 Let $X_u(\alpha,v) = (uv,ux_1,...,ux_p$). 
 Observe that $m_\alpha(v)\not= m_\alpha(u)$, otherwise $X_v(\alpha,u)$ and $X_u(\alpha,v)$ contain only 
 the edge $uv$ and are not cycles in $D_v({\alpha})$ and $D_u(\alpha)$, respectively.  Also note that 
 $m_\alpha(x_p) = m_\alpha(y_q) = \alpha(uv)$. The proposition of the lemma is evident if $p=1$. 
 We consider now the case $p\geq 2$.

Since $X_u(\alpha,v)$ is saturated, the component of $K_\alpha(\alpha(uv), m_\alpha(u))$ containing $u$ also contains the vertex $x_p$.
 By definition of $X_v(\alpha,u)$, $\alpha(vy_1) = m_\alpha(u)$ and thus $y_1$ is in the same component  of 
$K_\alpha(\alpha(uv),m_\alpha(u))$ as $u$ and $x_p$.   Denote this component by $P$. Clearly, $P$  contains the edges $uv$ and  $vy_1$ but not  $ux_p$. 

We will  show that $\{x_1,...,x_p\}\cap \{y_1,...,y_q\}=\emptyset$. 
This is evident if $G$ is triangle-free.
Assume now that $G$ is chordless.
 The path $P$ together with the edge $ux_p$ forms a cycle in $G$. This implies that $x_p\not= y_j$ for each $j=2,...,q$,  since otherwise the edge $vx_p$ is a chord 
 in the cycle $P\cup \{ux_p\}$. Furthermore,  $x_i\not= y_1$ for each $i=1,...,p-1$, since otherwise the edge $uy_1$ is a chord 
 in the cycle $P\cup \{ux_p\}$. Finally, $x_i\not= y_j$ for each pair $i,j$, where $1\leq i\leq p-1$ and  $2\leq j\leq q$, since otherwise the edge $uv$ is a chord  in the cycle induced by the set of edges $(E(P)\setminus \{uv\})\cup \{vx_i,x_iu,ux_p\}$.
 Thus  $\{x_1,...,x_p\}\cap \{y_1,...,y_q\}=\emptyset$.

Since $x_p\not=y_q$, $m_\alpha(y_q)=\alpha(uv)$ and $P$ contains $u, y_1$ and $x_p$, the path $P$  does not contain $y_q$. 
Let $C$ be the component of $K_\alpha(\alpha(uv),m_\alpha(u))$ containing $y_q$. 
Clearly, 
 $E(C)\cap \{uv,ux_1,...,ux_p,vy_1,...,vy_q\}=\emptyset$ 
and neither endpoint of $C$ is incident to an edge in 
$\{uv,ux_1,...,ux_p,vy_1,...,vy_{q-1}\}$,
as the only vertices missing colors  $\alpha(uv)$ or $m_\alpha(u)$ in 
$\{u,v,x_1,...,x_p,y_1,...,y_q\}$ are by definition $u, x_p$ and $y_p$, since $X_u(\alpha,v)$ 
and $X_v(\alpha,u)$ are cycles in $D_u(\alpha)$ and $D_v(\alpha)$, respectively. We consider the coloring $\alpha_1$ obtained from $\alpha$ by swapping $C$. 
We have that
 $\alpha_1(ux_i) = \alpha(ux_i)$ and $m_{\alpha_1}(x_i) = m_{\alpha}(x_i)$  for every  $i=0,1,...,p$;
  similarly we have  that   $\alpha_1(vy_j) = \alpha(vy_j)$  for every $y_j$ with 
 $0\leq j\leq q-1$.   
  Since the color $m_\alpha(u)$ is missing now at $y_q$, we  put  $m_{\alpha_1}(y_q) = m_\alpha(u)$.

Clearly,  the sequence $X_u(\alpha_1,v)$  is a  cycle in $D_u(\alpha_1)$.
 However, $X_u(\alpha_1,v)$ may not saturate in $\alpha_1$. 
 We distinguish the two cases.\medskip

{\bf Case 1}.  $X_u(\alpha_1,v)$ is not saturated in $\alpha_1$. 
By Lemma 2.2, the colorings $\alpha_1$ and  $X_u^{-1}(\alpha_1,v)$ are $K$-equivalent.  By
swapping $C$ for the second time (remember that $E(C)\cap \{uv,ux_1,...,ux_p,vy_1,...,vy_q\}=\emptyset$, and that
neither endpoint of $C$ is incident to an edge in $\{uv,ux_1,...,ux_p,vy_1,...,vy_{q-1}\}$), we obtain $X_u^{-1}(\alpha,v)$, hence 
the conclusion.

{\bf Case 2}.   $X_u(\alpha_1,v)$ is saturated in $\alpha_1$. 
Hence the component of the subgraph $K_{\alpha_1}(m_{\alpha_1} (u), m_{\alpha_1} (v))$
containing $u$ also contains $v$ thus does not contain $y_q$, since $m_{\alpha_1} (y_q ) = m_{\alpha_1} (u)$.
Let $C'$ be the component of $K_{\alpha_1}(m_{\alpha_1} (u), m_{\alpha_1} (v))$ containing $y_q$. 
Similarly as for $C$, we note that 
 $E(C')\cap \{uv,ux_1,...,ux_p,vy_1,...,vy_q\}=\emptyset$ 
and neither endpoint of $C'$ is incident to  edges in $\{uv,ux_1,...,ux_p,vy_1,...,vy_{q-1}\}$.

 We consider the coloring $\alpha_2$ 
 obtained from $\alpha_1$ by swapping $C'$. Clearly, we can put now $m_{\alpha_2}(y_q )=m_{\alpha_1}(v)$.
 Then the sequence $X_v(\alpha_2,u)=(vu,vy_1,...,vy_q)$ is a path in $D_v(\alpha_2)$.
  Let 
 $\alpha_3=X_v^{-1}(\alpha_2,u)$.  Then, by Definition 2,  $m_{\alpha_3}(u)=\alpha(uv)$, $m_{\alpha_3}(v)=\alpha(uv)$
 and $\alpha_3(uv)=\alpha_2(vy_1)=m_\alpha(u)$.
 By Lemma 2.1,  the colorings $\alpha_2$ and  $\alpha_3$ are $K$-equivalent.  Note that  
  $\alpha_3$ assigns the color $\alpha(uv)$ to no edge in $ \{uv,ux_1,...,ux_p,vy_1,...,vy_q\}$
  and $\alpha_3(ux_1)=\alpha_2(ux_1)=m_\alpha(v)$.  

  Since $m_{\alpha_3}(u)=\alpha(uv)=m_{\alpha_3}(x_p)$,  the sequence  $X_u(\alpha_3,x_1)=(ux_1,...,ux_p)$ is a path in $D_u(\alpha_3)$.
   Let $\alpha_4=X_u^{-1}(\alpha_3,x_1)$. In the coloring $\alpha_4$, we have $m_{\alpha_4}(v)=\alpha(uv)$
   and $m_{\alpha_4}(u)=m_{\alpha}(v)$, with $\alpha_4(uv)=m_\alpha(u)$.
 By Lemma 2.1,  $\alpha_3$ and  $\alpha_4$ are $K$-equivalent.   
 Since  $\alpha_4(uv) = m_\alpha(u)$,  there is a 
 unique  component of $K_{\alpha_4}(m_\alpha(u),m_\alpha(v))$ containing vertices of $C'$, which is precisely $C'\cup \{uv, vy_q\}$.

In the coloring $\alpha_5$ obtained from $\alpha_4$ by swapping $C'\cup \{uv,vy_q\}$,
put $m_{\alpha_5}(u)=m_\alpha(u)$.  There is a unique component of 
$K_{\alpha_5}(\alpha(uv),m_\alpha(u))$ containing vertices of $C$, which is precisely $C\cup \{vy_q\}$. 
Moreover, in the coloring $\alpha_5$, the sequence
$X_v(\alpha_5,y_1)=(vy_1, vy_q, vy_{q-1}, . . . , vy_2)$ is a cycle  in $D_v(\alpha_5)$. 
The cycle is not saturated since the component 
of $K_{\alpha_5}(\alpha(uv),m_\alpha(u))$ containing vertices of $C$ is precisely $C\cup \{vy_q\}$: since $q\geq 2$, it does not contain $y_1$. 
We consider the coloring $\alpha_6=X_v^{-1}(\alpha_5,y_1)$.
By Lemma 2.2,  
$\alpha_5$ and $\alpha_6$ are $K$-equivalent. Note that in 
$\alpha_6$, the component of $K_{\alpha_6}(\alpha(uv), m_\alpha(u))$ containing vertices of $C$ is precisely C: we swap it and obtain 
that $\alpha$ and  $X_u^{-1}(\alpha,v)$ are $K$-equivalent, as desired.
 \end{proof}

 \section{Proofs of main results}

 \begin{proof}[{\bf Proof of Theorem 1.2.}]
 The theorem is true 
  for  all Class 1  graphs with  maximum degree $2$.
  Assume that
 the theorem is true for all Class 1 graphs  with maximum degree $\Delta-1$ which are   triangle-free or chordless,  $\Delta\geq 3$. 
 
  Let $G$ be a  Class 1 graph  with maximum degree $\Delta$ that  is triangle-free or chordless, and let  
$\gamma$ be a   $\Delta$-coloring and $\beta$ a $(\Delta+1)$-coloring of $G$.

Denote by  $\cal C$  the set of all  $(\Delta+1)$-colorings of $G$ that are $K$-equivalent to  $\beta$. 

If there is a coloring  $\beta'  \in \cal C$  with $M(\beta',1) = M(\gamma,1)$, then the graph $G'=G-M(\gamma,1)$ 
has the maximum degree $\Delta-1$ and is Class 1. Therefore,
by the induction hypothesis,  the restrictions of $\beta'$ and $\gamma$ to $G'$ can be transformed to each other
by a series of $K$-changes.

Suppose now that $M(\beta',1)\not= M(\gamma,1)$ for each $\beta'  \in \cal C$.
We will show that this leads to a contradiction. We essentially follow the outline of \cite{bonamy23}.

In a given $(\Delta+1)$-coloring from $\cal C$, we say an edge is:
\begin{itemize}	
\item good if  it belongs to $M(\gamma,1)$ and is colored 1, 
\item bad if  it belongs to $M(\gamma,1)$ but is not colored 1, 
\item ugly if  it is not in $M(\gamma,1)$ but is colored 1.
\end{itemize}

Consider in $\cal C$ a coloring $\alpha$ which minimizes the number of ugly edges among the colorings  in $\cal C$
that minimize the number of bad edges. We call such a coloring {\it minimal}.
With each vertex $x\in V(G)$ we associate a  color from the set $\{1,2,...,\Delta+1\}$, denoted
$m_{\alpha}(x)$, that is missing  in $\alpha$ at $x$ with an additional 
	condition  that  $m_\alpha(x)=1$ if and only if  $1$ is the only color missing at $x$. 
	
	A vertex $x$ is called {\it free} if $m_\alpha(x)=1$.

	Note some  properties of the coloring $\alpha$. The first two of them are evident.
	
\begin{property} 
\label{property1}	
The degree of every free vertex of $G$  is $\Delta$.
\end{property}

\begin{property} 
\label{property2}
Every bad edge  is adjacent to an ugly edge.
\end{property}

\begin{property} [\cite{bonamy23}]
\label{property3} 
If $uv$ is an ugly edge, then the sequences $X_u(\alpha,v)$ and $X_v(\alpha,u)$ are cycles in $D_u(\alpha)$ and
$D_v(\alpha)$, respectively.
\end{property}

\begin{proof}
Since $uv$ is ugly,  $\alpha(uv)=1\not=\gamma(uv)$. Let $X_u(\alpha,v) = (ux_0, . . . , ux_p)$  and $m_\alpha(x_p)\not=1$, that is,
$X_u(\alpha,v)$ is not a cycle in $D_u(\alpha)$. We show that then there is a  
 $(\Delta+1)$-coloring $\alpha'$ of $G$ such that $\alpha$ and $\alpha'$ are  $K$-equivalent 
 and
  \begin{equation}
M(\alpha',1)=M(\alpha,1)\setminus\{ux_0\}.
\end{equation}
The following cases are possible:

Case 1.  $X_u(\alpha,v)$ is a path  in $D_u(\alpha)$, that is, the color $m_\alpha(x_p)$ is missing in $\alpha$ at $x_p$.
Then  $\alpha'=X_u^{-1}(\alpha,v)$ is $K$-equivalent to $\alpha$ and satisfies (1).\medskip

Case 2. $X_u(\alpha,v)$ is  a comet in $D_u(\alpha)$, that is,  $m_\alpha(x_p) = m_\alpha(x_{q-1}) = \alpha(ux_q)$
for some $0<q<p$. We swap the component $C$ 
  of $K(m_\alpha(u),\alpha(ux_q))$ containing the edge $ux_q$, and denote by $\alpha_1$ the resulting coloring. 
  
   If $C$ does not contains the vertex $x_{q-1}$, then $\alpha_1(ux_i)=\alpha(ux_i)$ for $i=0,...,q-1$,
  and the color $\alpha(ux_q)$ is missing in $\alpha_1$ at $u$. Put $m_{\alpha_1}(u)=\alpha(ux_q)$ and 
  $m_{\alpha_1}(x_{i})=m_\alpha(x_i)$ 
  for $i=1,...,q-1$. Then  the sequence $X_u(\alpha_1,v) = (ux_0, . . . , ux_{q-1})$ is a path in $D_u(\alpha_1)$
 and  $\alpha'=X_u^{-1}(\alpha_1,v)$  is $K$-equivalent to $\alpha$ and satisfies (1). 

 If   $C$ contains the vertex $x_{q-1}$, then $\alpha_1(ux_q)=m_\alpha(u)$ and $\alpha_1(ux_i)=\alpha(ux_i)$ for all $i\not=q$,
 $1\leq i\leq p$. Furthermore,
   the color $\alpha(ux_q)=m_\alpha(x_p)$ is missing at $u$ and the color $m_\alpha(u)$ is missing at $x_{q-1}$. 
   Put $m_{\alpha_1}(u)=m_\alpha(x_p)$,
$m_{\alpha_1}(x_{q-1})= m_{\alpha}(u)$   and $m_{\alpha_1}(x_i)=m_{\alpha}(x_i)$ for all $i\not=q-1$,
 $1\leq i\leq p$.
    Then  the sequence $X_u(\alpha_1,v) = (ux_0, . . . , ux_{p})$ is a path in $D_u(\alpha_1)$
 and  the coloring $\alpha'=X_u^{-1}(\alpha_1,v)$  is $K$-equivalent to $\alpha$ and satisfies (1).

 In all these cases the coloring $\alpha'$ has at most as many bad edges as $\alpha$, and fewer ugly edges.
This contradicts the minimality of $\alpha$.  

In the same way we can prove that $X_v(\alpha,u)$ is a cycle in $D_v(\alpha)$
\end{proof}

Property 3.3  implies  the following:

\begin{property}
\label{property4}
 If $e$ is an ugly edge, then both endpoints of $e$ have free neighbors.
 \end{property}
 
 \begin{property}  [\cite{bonamy23}]
 \label{property5}
 There is a bad edge with a free endpoint.
 \end{property}

\begin{proof}
Consider a bad edge (in $\alpha$) which, by Property 3.2, is adjacent to an ugly edge $e$. By Property 3.4, there exists some 
free vertex $u$ adjacent to an endpoint of $e$. Thus $m_\alpha(u)=1$. Then, by Property 3.1, 
 the degree of $u$ is $\Delta$ and 
therefore, there is an edge $uv\in M(\gamma,1)$  incident to $u$. 
Since $u$ is free, $uv$ is bad and, by Property 3.2, is  adjacent to an ugly edge. Since $u$ cannot be incident to
 an ugly edge (it is free), there is some vertex $w\in N(v)$ such that $vw$ is ugly. 
\end{proof}

\begin{property}
\label{property6}
  If $G$ is a chordless  graph, $uv$ is an ugly edge and $X_u(\alpha,v)=(ux_0,ux_1,...,ux_p)$, 
  where $ux_0=uv$ and $p\geq 1$, 
  then $vx_p\notin E(G)$.
 \end{property}
\begin{proof}
Suppose that $vx_p\in E(G)$. 
By Property 3.3, $X_u(\alpha,v)$ is a cycle in $D_u(\alpha)$.
Then $\alpha(ux_0)=1$ and  $m_\alpha(x_p)=1\not=m_\alpha(x_i)$ for $i=0,...,p-1$.
Consider the component $C$ of $K_\alpha(m_\alpha(u),\alpha(ux_p))$ containing the edge $ux_p$.

We will show that $C$ does not contain the vertex $x_{p-1}$.

This is true if $p\geq 2$,   since otherwise $G$ contains a cycle with a chord:
if the vertex $v$ is in $C$, the cycle $C\cup \{ux_{p-1}\}$ has the chord $uv$,
and if  $v$ is not in $C$, the cycle induced by  $(E(C)\setminus \{ux_p\})\cup \{ux_{p-1},uv,vx_p\}$ has the chord $ux_p$.

 If  $p=1$, then $vx_1$ is not in $C$, since otherwise $\alpha(vx_1)=m_\alpha(u)$ and  the coloring $X_v^{-1}(\alpha,u)$
has fewer bad edges as $\alpha$. This implies that $C$ does not contain the vertex $v=x_{0}$, since otherwise $vx_1$ is a chord for the cycle $C\cup \{uv\}$.

Let $\alpha'$ be the coloring of $G$ obtained from $\alpha$ by swapping  $C$. 
Since the color $\alpha(ux_p)=m_\alpha(x_{p-1})$ is  missing  in $\alpha'$ at $u$,  we put $m_{\alpha'}(u)=m_\alpha(x_{p-1})$.
We have that
$\alpha'(ux_i)=\alpha(ux_i)$ and we put  $m_{\alpha'}(x_{i})=m_\alpha(x_{i})$ for $i=0,...,p-1$.
  Now   the sequence  
  $X_u(\alpha',v)=(ux_0,...,ux_{p-1})$ is a path in $D_u(\alpha')$.
Then the coloring $X_u^{-1}(\alpha',v)$ has at most as many bad edges as $\alpha$, and fewer ugly edges.
Since $X_u^{-1}(\alpha',v)$ is $K$-equivalent to $\alpha$, this contradicts the minimality of $\alpha$.
\end{proof}

We continue to prove  the theorem.
By Property 3.5, there is a bad edge, say $wv$, with a free endpoint,
 say $w$, and an ugly edge, say $uv$, incident to $v$, where $\alpha(uv)=1$.
Consider the sequences 
$X_u(\alpha,v)=(uv,ux_1,...,ux_p)$ and $X_v(\alpha,u)=(vu,vy_1,...,vy_q)$.
By Property 3.3, the sequence $X_v(\alpha,u)$ is a cycle in $D_v(\alpha)$, and 
$X_u(\alpha,v)$ is a cycle in $D_u(\alpha)$. 

We will show that the colorings $\alpha$ and $\alpha'=X_u^{-1}(\alpha,v)$ are $K$-equivalent.

This is evident if $p=1$, and  follows from Lemma 2.2, if $p\geq 2$ and the cycle $X_u(\alpha,v)$ is non-saturated.    
Suppose now that $p\geq 2$ and the cycle $X_u(\alpha,v)$ is saturated. 
Then  the component $P$ of $K_\alpha(1, m_\alpha(u))$ containing $u$  contains also the vertex $x_p$.
 We have that $\alpha(vu)=1$, $\alpha(vy_1) = m_\alpha(u)$ and thus the vertex $y_1$ is 
   in $P$ as $u$ and $x_p$.  
We also have that $m_\alpha(y_q)=1$ and, by Property 3.6, $vx_p\notin E(G)$. This implies that $y_q\not=x_p$ and $q\geq 2$.
Then, by Lemmas  2.3, we derive that the colorings $\alpha$ and  
$\alpha'=X_u^{-1}(\alpha,v)$ are $K$-equivalent. 

Now
$\alpha'(uv)=m_\alpha(v)\not=1$, 
  the color 1 is missing at $v$ in $\alpha'$, and
  the coloring $\alpha'$ 
has at most as many bad and ugly edges as $\alpha$.
By Property 3.6,  $vx_p\notin E(G)$  and so the edge $uw$ does not appear in $X_u(\alpha,v)$. 
Thus the color 1 is missing  in $\alpha'$ at $w$ and at  $v$. Then we  recolor the edge $wv$ with color $1$ and obtain
a $(\Delta+1$)-coloring which is $K$-equivalent to $\alpha$ and has  fewer bad edges, a contradiction.  

This complete the proof of the theorem.
\end{proof}

\begin{proof}[{\bf Proof of Corollary 1.3.}]
If $G$ is Class 1, the result follows from Theorem 1.2.
Let $G$ be a  triangle-free Class 2 graph, $\alpha$ be a $(\Delta(G)+1)$-coloring of $G$ and  $u$  be a vertex of $G$ 
with maximum degree $\Delta(G)$.
Consider two disjoint copies $G_1$ and $G_2$ of $G$. Let $u_1$ and $u_2$ be the copies of $u$ in $G_1$ and $G_2$, respectively.
Then the triangle-free graph $H$ obtained from $G_1\cup G_2$ by adding the edge $u_1u_2$ has the maximum degree $\Delta(G)+1$.
The coloring $\alpha$  induces a $(\Delta(G)+1)$-coloring of $H$ as follows: for each edge $e\in E(G)$ color its copies 
in $G_1$ and $G_2$ with the color  $\alpha(e)$. Finally, color the edge $u_1u_2$ with the unique color missing in $\alpha$ at $u$.
Thus $\chi'(H)=\Delta(G)+1=\Delta(H)$, that is, $H$ is a  triangle-free Class 1 graph. Then, by Theorem 1.2, all $(\Delta(H)+1)$-colorings of $H$ are $K$-equivalent.
This implies that all  $(\chi'(G)+1)$-colorings of $G$ are $K$-equivalent, because $\Delta(H)+1=(\Delta(G)+1)+1=\chi'(G)+1$
and any series of $K$-changes in $H$ induces a series of $K$-changes in $G$.
\end{proof}

Corollary 1.4 follows from Theorem 1.2,  since $\chi'(G)=\Delta(G)$ for every simple chordless graph $G$ with $\Delta(G)\geq 3$
(see \cite{machado}).

The next proposition shows that there exists an  infinite class of chordless graphs that are not triangle-free. 

\begin{proposition}
For any two integers 
$k\geq 0$, $d\geq 1$, there is a simple chordless graph $H_k$ with  diameter  greater than $2^{k}d$ which contains at least
$3^{k}$ triangles.
\end{proposition}

\begin{proof}
We construct such a graph inductively.
 Choose  an arbitrary tree  $G$ of diameter $D(G)=d$ and two vertices $x_1$, $y_1$ in $G$ such that the distance $d_G(x_1,y_1)$ between them is $D(G)$. Let $H_0$ be the graph obtained from $G$ by adding  two new vertices $x_2, x_3$ and three new edges 
 $x_1x_2, x_2x_3, x_3x_1$. Clearly, $D(H_0)=d+1>2^0d$ and $H_0$ contains one triangle.
 
 Suppose we have already constructed a chordless graph $H_t$ with diameter $D(H_t)>2^td$ 
having at least $3^t$ triangles,  $t\geq 0$. Choose  two vertices $u$ and $v$ in $H_t$ such that the distance $d_{H_t}(u,v)$ between them  
is $D(H_t)$.
Let  
$H_t(1), H_t(2),H_t(3)$  be three disjoint copies of $H_t$, and let $u_i, v_i$ be the copies of  $u, v$ in $H_t(i)$, $i=1,2,3$.
Then the graph $H_{t+1}$ obtained from
$H_t(1)\cup H_t(2)\cup H_t(3)$ by adding the edge $u_1u_2, u_2u_3, u_3u_1$
 is  chordless, has 
more than $3^{t+1}$  triangles
and  $D(H_{t+1})>2^{t+1}d$ because $D(H_{t+1})\geq d_{H_{t+1}}(v_1,v_2)>d_{H_{t}}(v_1,u_1)+d_{H_{t}}(u_2,v_2)$.
\end{proof}

After finalizing our paper, we discovered that 
 J. Narboni  published  a preprint where he  proved  that all Class 1 graphs have the Vizing property
 (see J. Narboni, Vizing's conjecture holds. arXiv:2302.12914, 2023, 43 pages).

\end{document}